\documentclass[preprint,12pt,3p]{elsarticle}

\usepackage{amssymb}
\usepackage{amsmath,amsthm}
\usepackage{mathrsfs}
\usepackage{hyperref}
\usepackage{cleveref}
\usepackage{needspace}
\usepackage[all,cmtip]{xy}
\usepackage{epsf, graphicx}
\usepackage{latexsym,amsfonts,amsbsy,amssymb}
\usepackage{geometry}
\usepackage{titletoc}
\usepackage{color,xcolor}
\usepackage{rotating}
\usepackage{algorithm}
\usepackage{algorithmic}
\usepackage{amscd}

\makeatletter

\@addtoreset{equation}{section} \makeatother
\newtheorem{theorem}{Theorem}[section]
\newtheorem{lemma}[theorem]{Lemma}

\newtheorem{corollary}[theorem]{Corollary}
\newtheorem{remark}[theorem]{Remark}
\newtheorem{void}[theorem]{}
\newtheorem{proposition}[theorem]{Proposition}

\def\Irr{{\rm Irr}}
\def\Ind{{\rm Ind}}
\def\Res{{\rm Res}}

\def\End{{\rm End}}

\def\O{\mathcal{O}}

\def\FF{\mathbb{F}}

\def\Z{\mathbb{Z}}

\def\rk{{\rm rk}}

\makeatletter
\def\ps@pprintTitle{%
\let\@oddhead\@empty
\let\@evenhead\@empty
\def\@oddfoot{\reset@font\hfil\thepage\hfil}
\let\@evenfoot\@oddfoot
}
\newcommand{\dedication}[1]{\g@addto@macro\elsauthors{\par\medskip{\normalsize\itshape #1}\par}}
\makeatother

\begin{document}

\begin{frontmatter}
	
	\title{The Broué invariant of a Morita equivalence with an endopermutation source}

	\author{Xin Huang}

	\dedication{Dedicated to Robert Boltje on the occasion of his (65+1)th birthday}
	
	
	\begin{abstract}
A perfect isometry $I$ (introduced by Broué) between two blocks $b$ and $c$ is a frequent phenomenon in the block theory of finite groups. It maps an irreducible character $\psi$ of $c$ to $\pm$ an irreducible character of $b$. Broué proved that the ratio of the codegrees of $\psi$ and $I(\psi)$ is a rational number with $p$-value zero and that its class in $\mathbb{F}_p$ is independent of $\psi$. This element is called the Broué invariant of $I$ by Boltje.	The goal of this paper is to show that if $I$ comes from a Morita equivalence with an endopermutation source $V$, then, up to a sign, the Broué invariant of $I$ is determined by local data of $b$ and $c$. Therefore, up to a sign, it is independent of the endopermutation-source Morita equivalence. Moreover, we show that the sign factor is given by the reduction of the rank of $V$ modulo $p$. As a corollary, we obtain that the Isaacs--Navarro refinement of the Alperin--McKay conjecture holds for inertial blocks.
\end{abstract}
	
	
\end{frontmatter}


\section{Introduction}\label{s1}

Throughout this paper $p$ is a prime, $k$ is an algebraically closed field of characteristic $p$. In this section $\O$ is a complete discrete valuation ring with quotient field $K$ of characteristic $0$ and residue field $k$. We denote by $a\mapsto \bar{a}$ the canonical homomorphism $\O\to k$. Assume that $K$ is a splitting field for all finite groups considered below. 

Let $G$ and $H$ be finite groups, $b$ a block (idempotent) of $\O G$ and $c$ a block of $\O H$. Denote by $\Z\Irr(KGb)$ the group of generalized characters of $G$ over $K$ associated with the block $b$. For any $\chi\in \Irr(KGb)$, the integer $|G|/\chi(1)$ is called the {\it codegree} of $\chi$. Let $B=\O Gb$ and $C=\O Hc$ be block algebras, and assume that a $(B,C)$-bimodule $M$ induces a Morita equivalence between $B$ and $C$. The tensor functor $(K\otimes_\O M)\otimes_{KH}-$ induces a perfect isometry
$$I_M:\Z\Irr(KHc)\cong\Z\Irr(KGb)$$
sending any $\psi\in\Irr(KHc)$ to some $\chi\in \Irr(KGb)$. See \cite[D\'{e}finition 1.4]{Broue1990} or \cite[Definition 9.2.1]{Lin18b} for the definition of a perfect isometry. By \cite[Lemme 1.6]{Broue1990} or \cite[Corollary 9.2.4]{Lin18b}, the rational numbers
$$\frac{|G|/I_M(\psi)(1)}{|H|/\psi(1)},~~~\psi\in \Irr(KHc)$$ 
are units in the localization $\Z_{(p)}=\{\frac{n}{m}\mid m,n\in \Z,~p\nmid m\}$ and their residue classes in $\Z_{(p)}/p\Z_{(p)}=\FF_p$ are equal (independent of $\psi$). Following Boltje \cite[(1.1)]{Boltje}, we denote this element in $\FF_p^\times$, which is uniquely determined by $M$, by $\beta(M)$ and call it the {\it Brou\'{e} invariant} of $M$.

Let $(D,e)$ be a maximal $b$-Brauer pair. (See e.g. \cite[Section 6.3]{Lin18b} for Brauer pairs.) The block algebra $kC_G(D)e$ has the central defect group $Z(D)$ (see \cite[Corollary 6.3.10]{Lin18b}), and up to isomorphism, it has a unique simple module $S$ (see \cite[Proposition 6.6.5]{Lin18b}). Following \cite[(1.2)]{Boltje}, we set
$$b(B):=\frac{[C_G(D):Z(D)]}{\dim_k(S)}.$$
Since any two maximal $b$-Brauer pairs are $G$-conjugate, the rational number $b(B)$ is independent of the choice of $(D,e)$. Since the block algebra $kC_G(D)e$ has the central defect group $Z(D)$, the image of $e$ under the canonical homomorphism $kC_G(D)\to k(C_G(D)/Z(D))$ is a block of defect zero (see \cite[Proposition 6.6.5]{Lin18b}). Hence $b(B)$ is the codegree of the irreducible character of $C_G(D)/Z(D)$ and  an integer which is not divisible by $p$. Following Boltje, we denote by 
$$\beta(B):=\overline{b(B)}\in \FF_p^\times$$
the residue class of $b(B)$ in $\Z/p\Z=\FF_p$. It is an invariant of the block algebra $B$. The main result of this paper is the following:

\begin{theorem}\label{thm:main}
Let $G$ and $H$ be finite groups, $b$ a block of $\O G$ and $c$ a block of $\O H$. Write $B=\O Gb$ and $C=\O Hc$. Assume that there is a $B$-$C$-bimodule $M$ inducing a Morita equivalence between $B$ and $C$. Assume that when regarded as an $\O (G\times H)$-module, $M$ has an endopermutation source $V$.  Then
\begin{equation*}
\beta(M)=\overline{\rk_\O(V)}\frac{\beta(B)}{\beta(C)}.
\end{equation*}
In particular, up to a sign, $\beta(M)$ does not depend on $M$, because $\overline{\rk_\O(V)}=\pm1$; see e.g. \cite[Proposition 7.3.10 (i)]{Lin18b}. 
\end{theorem}

One of the keys to the proof of Theorem \ref{thm:main} is Theorem \ref{theorem:Green and slashed} below, which describes the connection between a Green correspondent and a slashed module of a Brauer-friendly module, and which may be of independent interest. The other keys mainly come from the references \cite{Boltje} and \cite{H24}. Following Puig \cite[2.16]{Puig11}, a block algebra $\O Gb$ with a defect group $D$ is called {\it inertial} if it is Morita equivalent to its Brauer correspondent $\O N_G(D)c$ via a bimodule with an endopermutation source. By Theorem \ref{thm:main}, we obtain the following corollary:

\begin{corollary}\label{cor:inertial-IN}
Let $G$ be a finite group, $b$ an inertial block of $\O G$ with a defect group $D$, $H=N_G(D)$, and $c$ the Brauer correspondent of $b$ in $\O H$. Then there is a sign $\varepsilon\in\{1,-1\}$ and a bijection
$$\alpha:\Irr(KHc)\longrightarrow\Irr(KGb)$$
which preserves heights and satisfies, for every $\psi\in\Irr(KHc)$ of height zero,
$$ \alpha(\psi)(1)_{p'}\equiv \varepsilon [G:H]_{p'}\,\psi(1)_{p'}\pmod p.$$
In particular, the Isaacs--Navarro refinement of the Alperin--McKay conjecture \cite[Conjecture B]{IN2002} holds for the block $b$.
\end{corollary}

\begin{proof}
Let $M$ be a bimodule with an endopermutation source $V$ inducing a Morita equivalence from $\O Hc$ to $\O Gb$, and let $\alpha$ be the height-preserving bijection induced by $I_M$. Since $\beta(\O Gb)=\beta(\O Hc)$, Theorem \ref{thm:main} gives $\beta(M)=\overline{\rk_\O(V)}=\overline\varepsilon$ for a sign $\varepsilon\in\{1,-1\}$. The definition of $\beta(M)$ and $|G|_{p'}/|H|_{p'}=[G:H]_{p'}$ yield the required congruence.
\end{proof}

\section{Connection between a slashed module and a Green correspondent}\label{sec:dimensions}
In this section $\O$ is a complete discrete valuation ring with residue field $k$; we allow the case $\O=k$.

\begin{void}
	{\rm \textbf{The Brauer construction.} Let $G$ be a finite group. We refer to \cite[Definition 1.3.1]{Lin18a} for the definition of a $G$-algebra. If $A$ is a $G$-algebra (resp. $\O G$-module), we denote by $A^H$ the $N_G(H)$-subalgebra (resp. $\O N_G(H)$-submodule) of $H$-fixed points of $A$ for any subgroup $H$ of $G$.  For any two $p$-subgroups $Q\leq P$ of $G$, the {\it relative trace map} ${\rm Tr}_Q^P:A^Q\to A^P$, is defined by ${\rm Tr}_Q^P(a)=\sum_{x\in [P/Q]}{}^xa$, where $[P/Q]$ denotes a set of representatives of the left cosets of $Q$ in $P$. We denote by $A(P)$ the {\it Brauer quotient} of $A$, i.e., the $N_G(P)$-algebra (resp. $kN_G(P)$-module)
		\[A^P/(\sum_{Q<P}{\rm Tr}_Q^P(A^Q)+J(\O)A^P),\]
		where $J(-)$ denotes the Jacobson radical.
		If $A$ is a $G$-interior algebra (for instance, $A={\rm End}_\O(M)$ for some $\O G$-module $M$), then the Brauer quotient $A(P)$ has a natural structure of $C_G(P)$-interior algebra. 
	}
\end{void}

\begin{theorem}[Dade's slashed modules; see e.g. {\cite[Proposition 7.3.7 (i)]{Lin18b}}]\label{theorem:Dade's slahsed modules}
	Let $P$ be a finite $p$-group and $V$ an endopermutation $\O P$-module having an indecomposable direct summand with vertex $P$. For any subgroup $Q$ of $P$, and any subgroup $H$ of $P$ satisfying $H\leq N_P(Q)$, there is up to isomorphism a unique endopermutation $kH$-module $V[Q]$ such that $(\End_\O(V))(Q)\cong \End_k(V[Q])$ as $H$-algebras. Moreover, $V[Q]$ has an indecomposable direct summand with vertex $H$.
\end{theorem}

\begin{remark}\label{remark:Defres}
	{\rm The $kH$-module $V[Q]$ in Theorem \ref{theorem:Dade's slahsed modules} is called a {\it $Q$-slashed module} attached to $V$ over the group $H$. 
	}
\end{remark}

A Brauer-friendly module is a direct sum of indecomposable modules with compatible fusion-stable endopermutation sources; we refer to \cite[Definition 8]{Biland} for its definition.

\begin{proposition}[{\cite[Proposition 2.6]{H24}}]\label{brauerfriendly}
	Let $G$ and $H$ be finite groups, and let $b$ and $c$ be blocks of $\O G$ and $\O H$ respectively. Let $M$ be an indecomposable
	$(\O Gb, \O Hc)$-bimodule with an endopermutation source inducing a stable equivalence of Morita type between $\O Gb$ and $\O Hc$. Then $M$ is a Brauer-friendly $\O(G\times H)(b\otimes c^\circ)$-module. Here $c^\circ$ denotes the image of $c$ under the anti-automorphism $\O H\to \O H$ sending any $h\in H$ to $h^{-1}$.
\end{proposition}

Note that although \cite[Proposition 2.6]{H24} is stated over $k$, the proof works over $\O$.

\begin{lemma}[{\cite[Theorems 18 and 21]{Biland}}]\label{slash}
	Let $G$ be a finite group, $b$ a block of $\O G$, $M$ a Brauer-friendly $\O Gb$-module, $(P,e)$ a $b$-Brauer pair, and $H$ a subgroup of $G$ such that $PC_G(P)\leq H\leq N_G(P,e)$. Denote by $\hat{e}$ the unique block of $\O C_G(P)$ that lifts $e$. Then there exists a Brauer-friendly $kHe$-module $Sl_{(P,e)}^H(M)$ and an isomorphism of $C_G(P)$-interior $H$-algebras
	$$\theta_{(P,e)}^H:~({\rm End}_\O(\hat{e}M))(P)\cong {\rm End}_k(Sl_{(P,e)}^H(M)).$$
\end{lemma}

Following Biland, the pair $(Sl_{(P,e)}^H(M),\theta_{(P,e)}^H)$ or just the $kHe$-module $Sl_{(P,e)}^H(M)$ is called a $(P,e)$-{\it slashed module} attached to $M$ over the group $H$. This is a generalization of Dade's slashed modules for endopermutation modules of finite $p$-groups.



\begin{lemma}\label{lem:source-sign}
Let $G$ be a finite group, $P$ a normal $p$-subgroup of $G$, and $V$ an $\O$-free indecomposable endopermutation $\O P$-module with vertex $P$. Suppose that an $\O$-free $\O G$-module $U$ satisfies $\Res_P^G(U)\cong V^{\oplus n}$ for some $n$. If $U_P$ is a $P$-slashed module attached to $\Res_P^G(U)$ over $P$, then
\begin{equation*}
 \dim_k(U_P)=n~~~{\rm and}~~~\rk_\O(U)=\rk_\O(V)\dim_k(U_P).
\end{equation*}
\end{lemma}

\begin{proof}
Let $V[P]$ be a $P$-slashed module attached to $V$ over $P$. By the definition of the slash construction, one easily checks that $V[P]^{\oplus n}$ is a $P$-slashed module attached to $\Res_P^G(U)$ over $P$. Hence by Theorem \ref{theorem:Dade's slahsed modules}, $U_P\cong V[P]^{\oplus n}$. It is well-known that $V[P]\cong k$; see e.g. \cite[the beginning of Section 4]{Biland} or \cite[Proposition 7.3.10 (i)]{Lin18b}. Hence $\dim_k(U_P)=n$. It follows that $\rk_\O(U)=n\rk_\O(V)=\rk_\O(V)\dim_k(U_P)$.
\end{proof}

The following is the main result of this section:

\begin{theorem}\label{theorem:Green and slashed}
Let $G$ be a finite group, $b$ a block of $\O G$ and $M$ an indecomposable $\O$-free Brauer-friendly $\O Gb$-module with a source triple $(P,e,V)$; see \cite[Definition 2]{Biland} for this notion. Let $M'$ be a Green correspondent of $M$ over $\O N_G(P)$ and let $M_0$ be the $\O N_G(P,e)$-module $eM'$. Let $M_P$ be a $(P,e)$-slashed module attached to $M$ over $PC_G(P)$. Then $$\rk_\O(M_0)=\rk_\O(V)\dim_k(M_P).$$
Note that $M_0$ is exactly a Green correspondent of $M$ with respect to the vertex subpair $(P,e)$ in the sense of \cite[Definition 2]{Biland}.
\end{theorem}

\begin{proof}
By the definition of a Brauer-friendly $\O Gb$-module (see \cite[Definition 8]{Biland}), $V$ is fusion-stable in the group $G$ with respect to the pair $(P,e)$ in the sense of \cite[Definition 5]{Biland}. Hence every conjugate of $V$ by an element of $N_G(P,e)$ is isomorphic to $V$. By \cite[Lemma 4.1(b)]{Boltje}, $M_0$ is an indecomposable $\O N_G(P,e)$-module with an $\O P$-source $V$. Since $M_0$ is a summand of $\Ind_P^{N_G(P,e)}(V)$ and $P$ is a normal subgroup of $N_G(P,e)$, the Krull--Schmidt theorem implies that
\begin{equation}\label{eq:homogeneous-green-1}
	\Res_P^{N_G(P,e)}(M_0)\cong V^{\oplus n}
\end{equation}
for some positive integer $n$.

By definition, one easily checks that $M_P$ is also a $(P,e)$-slashed module attached to the Green correspondent $M'$ over $PC_G(P)$; see \cite[Section 6]{Biland}. Hence by definition, one sees that $\Res_P^{PC_G(P)}(M_P)$ is a $P$-slashed module attached to $\Res_P^{N_G(P,e)}(M_0)$ over $P$. Applying Lemma \ref{lem:source-sign} to \eqref{eq:homogeneous-green-1}, we obtain
\begin{equation*}
	\rk_\O(M_0)=\rk_\O(V)\,\dim_k(M_P),
\end{equation*}
as required.
\end{proof}

\section{Proof of Theorem \ref{thm:main}}\label{sec:proof}

\begin{lemma}\label{lem:one-simple}
	Let $A$ and $B$ be split finite-dimensional $k$-algebras, each having only one simple module, say $S$ and $T$, up to isomorphism. Suppose that an $(A,B)$-bimodule $M$ induces a Morita equivalence. Let $N$ be another $(A,B)$-bimodule, projective as a right $B$-module, such that $\dim_k(N)$ is a multiple of $\dim_k(M)$. Then, for every nonzero finite-dimensional $B$-module $L$,
	\begin{equation*}
		\frac{\dim_k(N\otimes_{B}L)}{\dim_k(L)}=\frac{\dim_k(N)\,\dim_k(S)}{\dim_k(M)\,\dim_k(T)}.
	\end{equation*}
\end{lemma}

\begin{proof}
	There is, up to isomorphism, just one indecomposable projective right $B$-module $Q$. Since $B$ is split, we have $\dim_k(Q\otimes_{B}T)=1$. The isomorphism $M\otimes_{B}T\cong S$ therefore implies $M\cong Q^{\oplus\dim_k(S)}$ as right modules. Write $N\cong Q^{\oplus n}$. Comparing dimensions gives $n/\dim_k(S)=\dim_k(N)/\dim_k(M)$, and $\dim_k(N\otimes_{B}T)=n$. Since $N$ is a projective right $B$-module, the tensor functor $N\otimes_{B}-$ is exact. If $L$ has composition length $\ell$, then $\dim_k(L)=\ell\dim_k(T)$ and $\dim_k(N\otimes_{B}L)=\ell n$. This completes the proof.
\end{proof}

Now we start to prove Theorem \ref{thm:main}. So we start to use the notation and assumptions of Theorem \ref{thm:main}. We divide the proof in several parts. 

\begin{void}
{\rm \textbf{Standard notation for an endopermutation-source Morita equivalence.} Since $M$ has an endopermutation source, the structure theorem for such equivalences (see \cite[7.6]{Puig1999} or \cite[Theorem 9.11.2]{Lin18b}) allows us to choose maximal Brauer pairs $(D,e)$ for $B$ and $(E,f)$ for $C$, a group isomorphism $\varphi:E\cong D$, and an indecomposable endopermutation source $V'$ of $M$ with vertex
$$X=\Delta(D,\varphi,E)=\{(\varphi(u),u)\mid u\in E\}.$$
Since $V$ and $V'$ have the same $\O$-rank, we may assume that $V=V'$. The isomorphism $\varphi$ induces an isomorphism between the fusion systems on $D$ and $E$ determined by $(D,e)$ and $(E,f)$, and the source $V$ is fusion-stable; see \cite[Theorem 9.11.2]{Lin18b}. Denote by $\hat{e}$ the unique block of $\O C_G(D)$ that lifts $e$ and $\hat{f}$ the unique block of $\O C_H(E)$ that lifts $f$.
Denote by $p_1:G\times H\to G$ and $p_2:G\times H\to H$ the canonical projections.
	
Set	
\[I=N_G(D,e),\qquad J=N_H(E,f),\qquad
	Y=N_{G\times H}(X,e\otimes f^\circ).
	\]
	Since the block fusion systems determined by the maximal Brauer pairs $(D,e)$ and $(E,f)$ are isomorphic, we have
	\begin{equation}\label{eq:normalizers}
		\begin{gathered}
			p_1(Y)=I,\quad p_2(Y)=J,\\
			Y\cap(G\times1)=C_G(D)\times1,\qquad
			Y\cap(1\times H)=1\times C_H(E),\\
			[I:C_G(D)]=[J:C_H(E)].
		\end{gathered}
	\end{equation}
	Indeed, $Y$ consists of the pairs in $I\times J$ which induce matching automorphisms under $\varphi$. Both projections are surjective, and the kernels are the indicated centralizers.}
\end{void}

\begin{void}
{\rm \textbf{Comparing a Green correspondent and a slashed module.} Let $M'$ be a Green correspondent of $M$ over $\O N_{G\times H}(X)$, and set
$$M_0=(\hat{e}\otimes\hat{f}^\circ)M'.$$
By \cite[Lemma 4.1(b)]{Boltje}, $M_0$ is an indecomposable $\O Y$-module with an $\O X$-source $V$. By Proposition \ref{brauerfriendly}, $M$ is a Brauer-friendly $\O(G\times H)(b\otimes c^\circ)$-module. By \cite[Lemma 2.7]{H24}, $(X,\hat{e}\otimes\hat{f}^\circ, V)$ is a source triple of $M$ in the sense of \cite[Definition 2]{Biland}. 
Let $M_D$ be an $(X,e\otimes f^\circ)$-slashed module attached to $M$ over $XC_{G\times H}(X)=X(C_G(D)\times C_H(E))$. By \cite[Theorem 1.2 (a)]{H24}, $M_D$ induces a Morita equivalence between $kC_G(D)e$ and $kC_H(E)f$. By Theorem \ref{theorem:Green and slashed} we have
	\begin{equation}\label{eq:green-slash-rank}
		\rk_\O(M_0)=\rk_\O(V)\,\dim_k(M_D).
	\end{equation}
	By \cite[Proposition 7.3.10 (i)]{Lin18b}, $\rk_\O(V)^2\equiv1\pmod p$.
	
	Let $N$ be the restriction of $M_0$ to $C_G(D)\times C_H(E)$, viewed as an $(\O C_G(D)\hat{e},\O C_H(E)\hat{f})$-bimodule. By Mackey's formula, the restriction of $\Ind_X^Y(V)$ to $1\times C_H(E)$ is free, because $X\unlhd Y$ and $X\cap(1\times C_H(E))=1$. Therefore its summand $N$ is a projective right $\O C_H(E)$-module. Thus $\bar{N}:=k\otimes_\O N$ is a projective right $kC_H(E)$-module, and \eqref{eq:green-slash-rank} implies $\dim_k(\bar{N})=\rk_\O(V)\,\dim_k(M_D)$. Let $S$ be the unique simple $kC_G(D)e$-module and $T$ the unique simple $kC_H(E)f$-module, up to isomorphism. Now by Lemma \ref{lem:one-simple}, for any finitely generated $\O$-free $\O C_H(E)\hat{f}$-module $L\neq 0$, we have
	\begin{equation}\label{eq:local-rank-ratio}
		\frac{\rk_\O(N\otimes_{\O C_H(E)\hat{f}}L)}{\rk_\O(L)}
		=\frac{\rk_\O(V)\dim_k(S)}{\dim_k(T)}.
	\end{equation}
This is an analogue of \cite[equation (4.12)]{Boltje}.}
\end{void}

\begin{void}
{\rm \textbf{The $p'$-part of the degree of a height zero character.} For a positive integer $n$, write $n=n_p n_{p'}$, where $n_p$ is a power of $p$ and $p\nmid n_{p'}$. Let $\psi\in\Irr(KHc)$ be an irreducible character of height zero, and let $L$ be an $\O$-free $\O Hc$-module affording $\psi$. Let $\chi=I_M(\psi)$. The perfect isometry preserves heights, and hence $\chi$ is of height zero as well. So we have
	\begin{equation*}
		\psi(1)_p=[H:E]_p~~~{\rm and}~~~\chi(1)_p=[G:D]_p.
	\end{equation*}
	Every $\O$-free indecomposable $\O H$-module with vertex $Q$ has rank divisible by $[H:Q]_{p}$; see e.g. \cite[Lemma 4.1 (a)]{Boltje}. It follows that $L$ has vertex $E$. Let $L'$ be a Green correspondent of $L$ over $\O N_H(E)$, and set $L_0=\hat{f}L'$, an $\O$-free $\O J\hat{f}$-module. By \cite[Lemma 4.1(b)]{Boltje}, $L'\cong\Ind_J^{N_H(E)}(L_0)$. Hence by the definition of a Green correspondence, we have
	$$\Ind_J^H(L_0)\cong L\oplus L_{\rm small},$$
	where each summand of $L_{\rm small}$ has a vertex conjugate to a proper subgroup of $E$. Hence by \cite[Lemma 4.1 (a)]{Boltje}, $\rk_\O(L_{\rm small})$ is divisible by $p[H:E]_p$. Consequently
	\begin{equation*}
		\psi(1)_{p'}=\frac{\rk_\O(L)}{\psi(1)_p}\equiv
		\frac{[H:J]\rk_\O(L_0)}{[H:E]_p}\pmod p.
	\end{equation*}
Since the left side of this congruence is not divisible by $p$, we have $\rk_\O(L_0)_p=[J:E]_p$ and
	\begin{equation}\label{eq:psi-congruence}
		\psi(1)_{p'}\equiv
		[H:J]_{p'}\,\rk_\O(L_0)_{p'}\pmod p.
	\end{equation}
}
\end{void}

\begin{void}	
{\rm \textbf{The $p'$-part of the degree of the image of a height zero character.} By \cite[Lemma 4.1(b)]{Boltje} and the definition of a Green correspondence, we have
$$\Ind_Y^{G\times H}(M_0)\cong M\oplus M_{\rm small},$$
where each indecomposable summand of $M_{\rm small}$ has a vertex conjugate to a proper subgroup of $X$. By \cite[Lemma 4.3]{Boltje}, the ranks of $M\otimes_{\O H}L_{\rm small}$ and $M_{\rm small}\otimes_{\O H}\Ind_J^H(L_0)$ are divisible by $p[G:D]_p$. Using $\Ind_J^H(L_0)\cong L\oplus L_{\rm small}$, we obtain
$$\begin{aligned}
	\chi(1)&=\rk_\O(M\otimes_{\O H}L)\\
	&\equiv\rk_\O\left(\Ind_Y^{G\times H}(M_0)\otimes_{\O H}\Ind_J^H(L_0)\right)
	\pmod{p[G:D]_p}.
	\end{aligned}$$
	
We are going to use \cite[Corollary 3.6]{Boltje} with the third group equal to $1$. Let $\mathcal{T}$ be a set of representatives of $J\backslash H/J$, chosen to contain $1$. For $h\in\mathcal T$, set
$${}^hJ=hJh^{-1},~~~C_h=C_H(E)\cap{}^hJ~~~{\rm and}~~~I_h=\{g\in I\mid (g,j)\in Y \text{ for some }j\in{}^hJ\}.$$
Define the $\O[{}^hJ]$-module ${}^hL_0$: its underlying $\O$-module is $L_0$, and any $j\in{}^hJ$ acts on it as $h^{-1}jh$ acts on $L_0$. Consider the $\O C_G(D)$-module
$$F_h=\Res_{C_G(D)\times C_h}^{Y}(M_0)\otimes_{\O C_h}\Res_{C_h}^{{}^hJ}({}^hL_0).$$
The module $F_h$ has an $\O I_h$-module structure given by
	$$g(m\otimes\ell)=(g,j)m\otimes j\ell,$$
for any  $g\in I_h$, $m\in M_0$ and $\ell\in {}^hL_0$, where $j\in{}^hJ$ is chosen so that $(g,j)\in Y$. Two such choices differ by an element of $C_H(E)\cap{}^hJ=C_h$, so the balancing relation over $\O C_h$ makes this action independent of the choice of $j$. If we use the terminology and notation used in \cite[3.1]{Boltje}, the $\O I_h$-module $F_h$ is exactly the extended tensor product $M_0 \overset{Y,~{}^hJ}{\otimes} {}^hL_0$. Since $p_2(Y)=J$, \cite[Corollary 3.6]{Boltje} gives an isomorphism of $\O G$-modules
	\begin{equation}\label{eq:mackey-induced-tensor}
\Ind_Y^{G\times H}(M_0)\otimes_{\O H}\Ind_J^H(L_0)\cong
			\bigoplus_{h\in\mathcal T}\Ind_{I_h}^G(F_h).
	\end{equation}
Equation \eqref{eq:mackey-induced-tensor} is an analogue of \cite[equation (4.10)]{Boltje}. We now examine the terms of \eqref{eq:mackey-induced-tensor}. 

$\bullet$ If $h\notin N_H(E)$, let $S_h$ be an $\O({}^hE)$-source of ${}^hL_0$. Then $M_0$ is a summand of $\Ind_X^Y(V)$ and ${}^hL_0$ is a summand of $\Ind_{{}^hE}^{{}^hJ}(S_h)$. Hence $F_h$ is a direct summand of 
$$\Ind_X^Y(V) \overset{Y,~{}^hJ}{\otimes} \Ind_{{}^hE}^{{}^hJ}(S_h).$$ 
Apply \cite[Theorem 3.5]{Boltje} to this extended tensor product. Since $X\unlhd Y$ and ${}^hE\unlhd{}^hJ$, all the inducing subgroups in that formula are the subgroup
$$Q_h=\{g\in D\mid (g,u)\in X\text{ for some }u\in{}^hE\}=\varphi(E\cap{}^hE)<D.$$
{\emergencystretch=1em Consequently $F_h$ is relatively $Q_h$-projective, and every indecomposable summand of $\Ind_{I_h}^G(F_h)$ has a vertex conjugate to a subgroup of $Q_h$. By \cite[Lemma 4.1(a)]{Boltje},\par}
$$p[G:D]_p\mid\rk_\O\left(\Ind_{I_h}^G(F_h)\right).$$

$\bullet$ If $h\in N_H(E)$ but $h\notin J$, then $C_H(E)\le{}^hJ$, so $C_h=C_H(E)$. The right action of $\hat{f}$ on the first tensor factor of $F_h$ and the left action of ${}^h\hat{f}$ on the second factor are both the identity. Since $\hat{f}$ and ${}^h\hat{f}$ are distinct block idempotents of $\O C_H(E)$, they are orthogonal. The balancing relation therefore gives	
$$m\otimes\ell=m\hat{f}\otimes{}^h\hat{f}\ell=m\otimes\hat{f}\,{}^h\hat{f}\ell=0,$$
for any $m\in M_0$ and $\ell\in {}^hL_0$. Hence $F_h=0$.
	
$\bullet$ Thus, when $\O$-ranks are computed modulo $p[G:D]_p$, only the double coset $J1J$ contributes. For its representative $h=1$, we have $C_1=C_H(E)$ and $I_1=I$. In the notation above, we have
$$F_1=N\otimes_{\O C_H(E)\hat{f}}\Res_{C_H(E)}^J(L_0).$$
The corresponding summand in \eqref{eq:mackey-induced-tensor} is $\Ind_I^G(F_1)$, whose rank is $[G:I]\rk_\O(F_1)$. Combining this with the preceding reductions gives
	\begin{equation*}
		\chi(1)_{p'}\equiv
		\frac{[G:I]\rk_\O(F_1)}{[G:D]_p}\pmod p.
	\end{equation*}
	Again since the left side of this congruence is not divisible by $p$, we have $\rk_\O(F_1)_p=[I:D]_p$ and
	\begin{equation}\label{eq:chi-congruence}
		\chi(1)_{p'}\equiv
		[G:I]_{p'}\,\rk_\O(F_1)_{p'}\pmod p.
	\end{equation}
}
\end{void}

\begin{void}
{\rm\textbf{Calculating the Brou\'e invariant.}  Applying \eqref{eq:local-rank-ratio} to $\Res_{C_H(E)}^J(L_0)$, whose $\O$-rank is $\rk_\O(L_0)$, and using the expression for $F_1$ above, we obtain
	\begin{equation}\label{eq:ranks-final}
		\frac{\rk_\O(F_1)}{\rk_\O(L_0)}=\frac{\rk_\O(V)\,\dim_k(S)}{\dim_k(T)},~~~{\rm and~hence}~~~
		\frac{\rk_\O(F_1)_{p'}}{\rk_\O(L_0)_{p'}}=\frac{\rk_\O(V)\,\dim_k(S)_{p'}}{\dim_k(T)_{p'}}.
	\end{equation}
	The second equality uses $p\nmid \rk_\O(V)$. Thus
	\begin{align*}
		\beta(M)
		&=\overline{\frac{|G|_{p'}\psi(1)_{p'}}
			{|H|_{p'}\chi(1)_{p'}}}=\overline{\frac{|G|_{p'}[H:J]_{p'}\rk_\O(L_0)_{p'}}
			{|H|_{p'}[G:I]_{p'}\rk_\O(F_1)_{p'}}}=\overline{\frac{|I|_{p'}\dim_k(T)_{p'}}
			{|J|_{p'}\rk_\O(V)\dim_k(S)_{p'}}}\\
		&=\overline{\rk_\O(V)}^{-1}
		\overline{\frac{|C_G(D)|_{p'}\dim_k(T)_{p'}}
			{|C_H(E)|_{p'}\dim_k(S)_{p'}}}
		=\overline{\rk_\O(V)}^{-1}\frac{\beta(B)}{\beta(C)}.
	\end{align*}
	The second equality follows from \eqref{eq:psi-congruence} and \eqref{eq:chi-congruence}, and the third from \eqref{eq:ranks-final}. The penultimate equality uses \eqref{eq:normalizers}; the last uses the definitions of $\beta(B)$ and $\beta(C)$ and the fact that $b(B)$ and $b(C)$ are $p$-units. Since $\rk_\O(V)^2\equiv1\pmod p$, $\overline{\rk_\O(V)}^{-1}=\overline{\rk_\O(V)}$. This proves Theorem \ref{thm:main}.
	\qed }
\end{void}

\bigskip\noindent\textbf{Acknowledgements.} I wish to thank Jennifer Guerrero, Radha Kessar, John McHugh, Nariel Monteiro, Tim Wesley, Deniz Y\i lmaz, and all other organizers for their hospitality at the conference celebrating Robert Boltje's 65th birthday. The research was supported by NSFC (12501024, 12471016), Fundamental Research Funds for the Central Universities (CCNU24XJ028), China Postdoctoral Science Foundation (GZC20252006, 2025T001HB) and China Scholarship Council (202506770066).

\bigskip
{\footnotesize School of Mathematics and Statistics, Central China Normal University, Wuhan 430079, China
	
Email address: xinhuang@mails.ccnu.edu.cn}

\end{document}